\documentclass[fleqn, eqnumis, pdf]{hjms}
\usepackage{footnote}
\usepackage{multicol}
\usepackage{booktabs}
\usepackage[flushleft]{threeparttable}
\usepackage[font=small,labelfont=bf]{caption}
\begin{document}
\title{On Upper Bounds for the Depth of some Classes of Polyhedra}
\author{ Mojtaba~Mohareri\footnote{Department of Pure Mathematics, Ferdowsi University of Mashhad, P.O.Box 1159-91775, Mashhad, Iran,  \ Email: {\tt m.mohareri@stu.um.ac.ir } ORCID:0000-0002-6725-2160} and  Behrooz~Mashayekhy\footnote{Department of Pure Mathematics, Ferdowsi University of Mashhad, P.O.Box 1159-91775, Mashhad, Iran,  \ Email: {\tt bmashf@um.ac.ir } ORCID:0000-0001-5243-0641 } \footnote{Corresponding Author.}} 
\begin{abstract}
 In this  paper,  we present upper bounds for the depth of some classes of polyhedra, including: polyhedra with finite fundamental group, polyhedra $P$ with abelian or free $\pi_1 (P )$  and finitely generated $H_i (\tilde{P};\mathbb{Z})$, 2-dimensional polyhedra with abelian or free fundamental group, and 2-dimensional polyhedra with elementary amenable fundamental group $G$ with finite cohomological dimension $cd(G)$. Furthermore, we provide some examples to show that some of these bounds are sharp. 
\end{abstract}
\section{Introduction and Motivation}
In 1979, at the International Topological Conference in Moscow, K. Borsuk introduced the capacity and depth in the shape category of compacta, together with some open questions \cite{So} (see \cite{B4,Dy,MS} for basic notions and results of shape theory).

Recall that a \textit{domination} in a given category $\mathcal{C}$ is a morphism $f : X \to Y$ , $X, Y \in Obj \mathcal{C}$, for which there exists a morphism $g : Y \to X$ in $\mathcal{C}$ such that $f\circ g = \mathrm{id}_Y$. Then we say that $Y$ is dominated by $X$, and we write $Y \leqslant X$ or $X \geqslant Y$. Moreover, $X < Y$ will denote that $X \leqslant Y$ holds but $Y \leqslant X$ fails (see, for example, \cite{7}).

In the following, $\mathcal{C}$ is the homotopy category of CW-complexes and homotopy classes of cellular maps between them or the shape category of compacta (pointed or unpointed).

Following K. Borsuk (cf. \cite{So}), define the capacity $C(A)$ of an $A \in Obj \mathcal{C}$
as the cardinality of the class of isomorphism classes of all the $X \in Obj \mathcal{C}$
such that $X \leqslant A$.

A system $X_k  < \cdots  < X_1 <X_0 =A$, where $X_i \in Obj \mathcal{C}$ for $i =0, 1, \ldots , k$, is called
a chain of length $k$ for $A \in Obj \mathcal{C}$. The depth $D(A)$ of $A$ is the least upper
bound of the lengths of all chains for $A$. If this upper bound is infinite, we
write $D(A) = \mathcal{N}_0$ (cf. \cite{So}).

It is clear that $D(A)\leq C(A)$ for each $A\in Obj\mathcal{C}$.

In the following,  every polyhedron and  CW-complex $P$ is assumed to be finite and connected.  Also, every map between two CW-complexes is assumed to be cellular. Since every polyhedron is homotopy equivalent to a finite CW-complex of the same dimension, and vice versa, we use the terms ``polyhedron'' and ``finite CW-complex'' interchangeably. We assume that the reader is familiar with the basic notions and facts of homotopy theory.

In 1979, K. Borsuk stated the question: Is there a compactum with infinite capacity, but finite depth?  (see, for example, [\cite{So}, Question 8]).

In the above problem the notions of shape and shape domination can be replaced by the notions homotopy type and homotopy domination, respectively. Indeed, by the known results in shape theory (from \cite{Ed,Ha}, \cite[Theorems 2.2.6, and 2.1.6]{Dy}) we get that, for each polyhedron $P$, there is a 1-1 functorial correspondence between the shapes of compacta shape dominated by $P$ and the homotopy types of CW-complexes (not necessarily finite) homotopy dominated by $P$ (in both pointed and unpointed cases).

Recall that each space homotopy dominated by a polyhedron has the homotopy type of a CW-complex, not necessarily finite (see \cite{Wall}). Thus the Borsuk problem is equivalent to: Is there a polyhedra with infinite capacity, but finite depth?

In this paper we consider dominations of a polyhedron in the category of CW-complexes and homotopy classes of cellular maps between them. 

 In \cite{3}  D. Ko\l{}odziejczyk  showed that the answer to the Borsuk question is positive: For every non-abelian poly-$\mathbb{Z}$-group $G$ and an integer $n\geq 3$, there exists a polyhedron $P$ such that $\pi_1 (P)\cong G$, $\dim P=n$, $C(P)$ is infinite, but $D(P)$ is finite. Thus, there exist polyhedra with polycyclic or nilpotent fundamental groups with this property. She then proved in \cite{7} that every polyhedron with virtually polycyclic fundamental group has finite depth. Recently, in \cite{D2018}, Ko\l{}odziejczyk  showed the same for all 2-dimensional polyhedra whose fundamental groups are elementary amenable with finite cohomological dimension or limit groups. 
 
When it comes to computation of the capacity and depth of polyhedra,  we refer to \cite{Mo,Mo1,Ab}. In \cite{Mo}, it is proved that the capacity and depth of $\bigvee_{n\in I} (\vee_{r_n} S^n)$ is equal to $\prod_{n\in I}(r_n +1)$ and $\sum_{n\in I}r_n $, respectively, where $\vee_{r_n} S^n$ denotes the wedge sum of $r_n$ copies of $S^n$, $I$ is a finite subset of $\mathbb{N}$ and $r_n \in \mathbb{N}\cup \{ 0\}$. In \cite{Mo1}, it is shown that for $n,m\geq 1$, the capacity of $S^n\times S^m$ is equal to 4 if $n\neq m$ and it is equal to 3 if $n=m$ but $D(S^n \times S^m)=2$ in general. Also, we investigated the capacity and depth of  lens spaces and $\mathbb{Z}_n$-complexes, i.e., connected finite 2-dimensional CW-complexes with finite cyclic fundamental group $\mathbb{Z}_n$. Finally, in \cite{Ab} it is shown that the capacity and depth of every compact orientable surface of genus $g\geq 0$ is equal to $g+2$ and  the capacity and depth of a compact non-orientable surface of genus $g > 0$ is $[\frac{g}{2}] + 2$.
 
In the present paper, we give an upper bound for the depth of each of the following classes of polyhedra: polyhedra with finite fundamental group, polyhedra $P$ with abelian or free $\pi_1 (P )$  and finitely generated $H_i (\tilde{P};\mathbb{Z})$, where $\tilde{P}$ is the universal covering space of $P$,  2-dimensional polyhdera with abelian or free fundamental group, and 2-dimensional polyhdera with elementary amenable fundamental group $G$ with $cd(G)<\infty$. Furthermore, we provide some examples to show that some of such bounds are sharp. As we mentioned above, it has been shown in \cite{4,7,D2018} that the depth of each of the following classes of polyhedra is finite: polyhedra with finite fundamental group, polyhedra $P$ with abelian or free $\pi_1 (P )$ and finitely generated $H_i (\tilde{P};\mathbb{Z})$, where $\tilde{P}$ is the universal covering space of $P$,  2-dimensional polyhdera with abelian or free fundamental group, and 2-dimensional polyhdera with elementary amenable fundamental group $G$ with $cd(G)<\infty$. It should be noted that the presented proofs are not in such a way one can use to compute their depth. In the present paper, we prove the same facts in a different approach which is significant from two points of view: first, our proofs are explicit and straightforward, and second, they are such that one can derive upper bounds for the depth.  
\section{Preliminaries}
Let us recall some definitions.

A homomorphism $g :G \to H$ of groups is an r-homomorphism if there exists a homomorphism $f :H \to G$ such that $g\circ f = id_H$. Then $H$ is called an \textit{r-image} of $G$.

Let $H$ be a subgroup of a group $G$. Then a  homomorphism $r:G\to H$ is said to be a \textit{retraction} if the inclusion  homomorphism $i:H\hookrightarrow G$ is a right inverse of $r$, i.e. $r(x)=x$ for all elements $x\in H$. Then $H$ is called  a \textit{retract} of $G$. By a \textit{proper retract} of $G$, we mean a retract $H$ of $G$ such that $H\neq G$.

The following fact follows from the definition.
\begin{lemma}\label{Semi1}
If $H$ is an r-image of $G$, then  ${\rm im}(f)$ is a retract of $G$.
\end{lemma}
Suppose that $N\unlhd G$ and there is a subgroup $H$ such that $G=HN$ and $H\cap N=1$; then $G$ is said to be the \textit{(internal) semidirect product} of $N$ and $H$ and denoted by $G=H\ltimes N$ or $G=N\rtimes H$. 

 The lemma below can be easily  deduced from the definition. 
\begin{lemma}\label{Semi2}
Let $G$ be a group. Then a subgroup $H$ of $G$ is a retract of $G$ if and only if $G=H\ltimes N$ for a normal subgroup $N$ of $G$.
\end{lemma}
Let $H$ be a subgroup of a group $G$. Then a subgroup $K$ is called a \textit{complement} of $H$ in $G$ if $G=HK$ and $H\cap K=1$. By Lemma \ref{Semi2}, a subgroup $H$ of $G$ is a retract of $G$ if and only if it has a normal complement in $G$.

Recall that a group $G$ is called \textit{Hopfian} if every epimorphism $f :G\to G$ is an automorphism (equivalently, $N = 1$ is the only normal subgroup for which $G/N\cong G$). Finitely generated abelian groups and free groups of finite rank are examples of Hopfian groups. 

Elementary amenable groups is the smallest class of groups that contains all abelian and all finite groups, and is closed under extensions and directed unions (see  \cite[p. 223]{Bri}).
The \textit{Hirsch length}, $h(G)$, of an elementary amenable group is finite and equal to $n\geq 0$, if $G$ has a series $1=H_0 \lhd \cdots \lhd H_r =G$ in which the factors are either locally finite or infinite cyclic, and exactly $n$ factors are infinite cyclic. In all other cases, $h(G)=\infty$ (see  \cite[p. 223]{Bri}). Recall that a group is \textit{locally finite} if all its finitely generated subgroups are finite. 
 
By a \textit{cohomoligcal dimension} of a group $G$ we mean (see  \cite{Br})
$$cd(G)=\sup \{ n: H^n (G,M)\neq 0; \mathrm{for some}; \mathbb{Z}G-\mathrm{module} M\}.$$ 

A \textit{series} of a group $G$ is a chain of subgroups $G=G_0 > G_1 >\cdots > G_k=1$. If each $G_i$ is  normal  in $G$, it is called a \textit{normal series} of $G$. The length of the series is the number of strict inclusions. 

\section{Splitting normal series}
In this section, we introduce a special normal series of a group called  splitting normal series and determine  the maximum length of such series of some well-known groups.
\begin{definition}\label{def1}
Let $G$ be a group.
\begin{enumerate}
\item
By a  \textit{splitting normal series} of $G$ we mean a normal series $G=G_0 > G_1 > \cdots > G_k=1$ of $G$ in which each $G_i$ has a complement in $G$. We define $n_1(G)$ to be the maximum length of a splitting normal series of $G$.
\item
By a  \textit{retract series} of $G$ we mean a series  $G=G_0 > G_1 > \cdots > G_k=1$ of $G$ in which each $G_i$  is a retract of $G$.  We define $n_2(G)$ to be the maximum length of a retract series of $G$.
\item
We define $n_3(G)$ inductively: $n_3(1)=0$ and $n_3(G)=1+\sup \{ n_3 (H): H\; \text{is a proper retract of}\; G\}$.
\end{enumerate} 
\end{definition} 
In the following, we state relationships between $n_i (G)$'s for $i=1,2,3$.
\begin{proposition}\label{pro1}
Let $G$ be a group. Then $n_1 (G)=n_2(G)$. Moreover, if $n_1 (G)<\infty$, then $n_1 (G)=n_2 (G)=n_3 (G)$.
\begin{proof}
It is easy to see that a  splitting normal series gives rise to a retract series, so we get $n_1(G) \leq n_2(G)$. 

Assume that $G=G_0 > G_1 > \cdots > G_k=1$ is a  retract series of $G$. Let $N_i$ be a normal complement of $G_i$ in $G$.  Since $G_2 < G_1$ and $G=G_2 N_2$, we get $G_1 = G_2 (N_2 \cap G_1)$. Then $N_1(N_2 \cap G_1)$ is also a normal complement of $G_2$ in $G$, so we can redefine $N_2 = N_1(N_2 \cap G_1)$, etc. This leads to the splitting normal series $G=N_k  > \cdots >N_1 > N_0=1$ of $G$ (note that $N_1=N_2$ imples that $N_2 \cap G_1 \subseteq N_1 \cap G_1 =1$ which implies that $G_1=G_2$ which is a contradiction). Hence $n_2(G)\leq n_1(G)$ and we have $n_1(G)=n_2(G)$ as claimed.

Now assume that $n_1 (G)<\infty$. It can be easily seen that $n_3 (G)$ provides the smallest possible function of $G$, say $n(G)$, with this property: if $H$ is a proper retract 	of $G$, then $n(H)\lneqq n(G)$. Hence, $n_3 (G)\leq n_1 (G)$ for every group $G$. Moreover, since $n_1 (G)<\infty$,  there exists a splitting normal series $G=G_{n_1(G)} >  \cdots >G_1 > G_{0}=1$ of $G$ with the maximum length. Trivially, we have $n_3 (G_1)=1$. Then $n_3 (G_2)\geq 1+n_3 (G_1)=2$. Similarly, $n_3 (G_3)\geq 1+n_3 (G_2)\geq 3$. By continuing this process, we see that $n_3 (G)=n_3 (G_{n_1 (G)})\geq 1+n_3 (G_{n_1 (G)-1})\geq 1+n_1 (G)-1=n_1 (G)$. So  $n_3 (G)\geq n_1 (G)$, hence $n_1 (G)=n_3 (G)$.
\end{proof}
\end{proposition}
By the above proposition, when $n_1 (G)<\infty$, we have $n_1 (G)=n_2 (G)=n_3 (G)$. We denote this number by $sl(G)$, i.e., \textit{splitting length} of $G$. 
\begin{lemma}\label{lem1}
Let $G$ be a group with $sl (G)<\infty$. If $H$ is a retract of $G$, then $sl(H)\leq sl(G)$. Moreover, if $H$ is proper retract, then $sl(H)\lneqq sl (G)$.
\end{lemma}
\begin{proof}
Assume that $H=H_0 > H_1 > \cdots > H_{sl(H)}=1$ is a splitting normal series of $H$ with the maximum length. Since $H$ is a retract of $G$, it has a normal complement $N$ in $G$. If $K_i$ is a complement of $H_i$ in $H$, then  $K_i$ is a complement $NH_i$ in $G$ because $NH_iK_i=NH=G$ and $NH_i \cap K_i\subseteq H_i\cap K_i=1$. So $G=NH_0 >NH_1 >\cdots >NH_{sl(H)}=N$ is a  splitting normal series of $G$. Then $sl(G)\geq sl(H)$. Now let $H$ be a proper retract of $G$. Then $N$ is nontrivial,  and so $G=NH_0 >NH_1 >\cdots >NH_{sl(H)}=N>1$ is a splitting normal series of $G$ which implies that $sl(G)\geq sl(H)+1$. Thus, $sl(H)\lneqq sl(G)$.
\end{proof}

The following observation comes from the definition. 
\begin{lemma}\label{must}
Let $G$ and $H$ be two groups. Then
\begin{enumerate}
\item
$sl(G)=0$ if and only if $G=1$.
\item
If $G \cong H$, then $sl(G)=sl(H)$.
\end{enumerate}
\end{lemma}
\begin{proposition}\label{pro2}
Let $G$ be a group. 
\begin{enumerate}
\item
If $G$ is a free group of finite rank, then  $sl(G)=rank(G)$.
\item
If $G$ is a finitely generated abelian group, then $sl(G)$ is the number of nonzero direct summands  in the canonical form of $G$.
\item
If $G$ is a countable elementary amenable group with $cd(G)<\infty$, then $sl(G)=h(G)$, where $h(G)$ is the Hirsch length of $G$.
\end{enumerate}
\end{proposition}
\begin{proof}
\begin{enumerate}
\item
Assume that $G=G_0 > G_1 > \cdots >G_{i-1}>G_{i}>\cdots $ is an arbitrary retract series of $G$. Since each $G_i$  is a  retract of $G$, each $G_i$  is a proper retract of $G_{i-1}$. Then there is an epimorphism $r:G_{i-1}\to G_i$ such that $r\circ i={\rm id}_{G_i}$, where $i:G_i\hookrightarrow G_{i-1}$ is the inclusion homomorphism. We claim that $rank(G_i)\lneqq rank(G_{i-1})$. Suppose, on the contrary, that $rank(G_{i-1}) =rank(G_i)$. This implies that $G_{i-1}\cong G_{i}$. Now by the fact that an epimorphism between isomorphic Hopfian groups is an isomorphism, we get that $r:G_{i-1}\to G_i$ is an isomorphism. Then $i:G_i\hookrightarrow G_{i-1}$ is an isomorphism which is a contradiction to the fact that $G_i$ is a proper retract of $G_{i-1}$. Hence $rank(G_i)\lneqq rank(G_{i-1})$ as claimed. This implies that $G_{rank(G)}=1$. Thus, $sl(G)=rank(G)$.
\item
It is easily concluded from the fundamental theorem of finitely generated abelian groups. 
\item
Let $G$ be a  countable elementary amenable group with $cd(G)<\infty$. Since the class of elementary amenable groups is closed under taking subgroups, each retract of $G$ is also countable elementary amenable. Since $cd(G)<\infty$, then $h(G)<\infty$ (see  \cite[Theorem 5]{Hil}). It is known that $H\subseteq G$ implies that $cd(H)\leq cd(G)$ (see \cite{Br}). Thus, we have $cd(H)<\infty$, hence $h(H)<\infty$, also for each retract $H$ of $G$.

Assume that $G=G_0 > G_1 > \cdots >G_{i-1}>G_{i}>\cdots $ is an arbitrary  retract series of $G$. Since each $G_i$ is a retract of $G$, each $G_i$ is a proper retract of $G_{i-1}$. Then there is an epimorphism $r:G_{i-1}\to G_i$ such that $r\circ i={\rm id}_{G_i}$, where $i:G_i\hookrightarrow G_{i-1}$ is the inclusion homomorphism. We claim that $h(G_i)\lneqq h(G_{i-1})$. Suppose, on the contrary, that $h(G_{i-1}) =h(G_i)$. Let $N_i$ be the kernel of the retraction $r$, i.e., $N_i\lhd G_{i-1}$ satisfies $G_{i-1}=G_i N_i$ and $G_i \cap N_i=1$. Then we have $h(G_{i-1})=h(N_i)+h(G_i)$ which implies that $h(N_i)=0$. Hence $N_i=1$ (see \cite[Lemma 3]{D2018}), and so $G_{i-1}=G_i$ which is a contradiction to the fact that $G_i$ is a proper retract of $G_{i-1}$. Hence $h(G_i)\lneqq h(G_{i-1})$ as claimed. This implies that $G_{h(G)}=1$. Thus, $sl(G)=h(G)$.
\end{enumerate}
\end{proof}
\begin{remark}
In \cite[Ch. I, Theorem 5]{Hil} the assumption that $G$ is countable can be omitted (see remarks in \cite[Ch. I]{Hil}). Hence, it could also be omitted in Proposition \ref{pro2}. Anyway, in our applications we consider only finitely presented groups, which are always countable.
\end{remark}
\section{Main results}
For a polyhedron $P$, let $X\leqslant P$ denote that $X$ is homotopy dominated by $P$. Writing this we will have in mind a fixed domination $d_X : P \to X$ of $P$ over $X$, and a fixed inverse map $u_X : X \to P$ (i.e. $d_X u_X \simeq  \textrm{id}_X$). It is easily   seen that the map $k_X = u_X d_X : P \to P$ is an idempotent in the homotopy category of CW-complexes and homotopy classes of maps between them.  From now on ``dominated'' will always mean ``homotopy dominated''.
 
 In what follows, $\tilde{X}$ will denote, as usual notation, the universal covering space of $X$.  Let $f:X\to Y$ be a cellular map of CW-complexes such that $f(x)=y$, for vertices $x\in X$, $y\in Y$. Choose $\tilde{x}\in p^{-1}(x)$, $\tilde{y}\in q^{-1}(y)$, where $p:\tilde{X}\to X$ and $q:\tilde{Y}\to Y$ denote the covering maps. Then it is well-known that there exists a unique map $\tilde{f}:\tilde{X}\to \tilde{Y}$ such that $q\tilde{f}=fp$ and $f(\tilde{x})=\tilde{y}$.
 
 Note that if $X \leqslant P$, where $P$ is a polyhedron, then $\tilde{X}\leqslant \tilde{P}$. As a result, ${\rm im}\big( \pi_1 (u_X )\big)$, ${\rm im}\big( H_i (u_X )\big)$ and ${\rm im}\big( H_i (\tilde{u}_X )\big)$ are retracts of $\pi_1 (P)$, $H_i (P)$ and $H_i (\tilde{P})$, respectively, for all $i$, by Lemma \ref{Semi1}.  
 
 We now prove our first main result.
\begin{theorem}\label{MainTheorem}
Let $P$ be a polyhedron of dimension $n$. If $sl(\pi_1 (P))<\infty$ and $sl( H_i (\tilde{P}))<\infty$ for $i\geq 2$, then $D(P)\leq  sl\big( \pi_1 (P)\big)  +\sum_{i=2}^{n} sl\big(  H_i (\tilde{P})\big)$.
\end{theorem}
\begin{proof}
Consider the following chain of CW-complexes:
$$
\cdots < X_{j+1}<X_j <\cdots <X_3 <X_2 <X_1 < X_0 =P.
$$
 Let $d_{X_{j+1}}:X_j \to X_{j+1}$ and $u_{X_{j+1}} :X_{j+1}\to X_j$ be the domination of $X_j$ over $X_{j+1}$ and the converse map, i.e., $d_{X_{j+1}} u_{X_{j+1}}\simeq id_{X_{j+1}}$. Then $\pi_1 (d_{X_{j+1}}) \pi_1 (u_{X_{j+1}})= id_{\pi_1 (X_{i+1})}$ and  $H_i (\tilde{d}_{\tilde{X}_{j+1}}) H_i (\tilde{u}_{\tilde{X}_{j+1}})= id_{H_i (\tilde{X}_{j+1})}$, for $2\leq i\leq n$ (note that since $X_j \leqslant P$ , we have $\tilde{X}_j \leqslant \tilde{P}$. But $\dim (\tilde{P})=\dim (P)$, and so $H_i (\tilde{P})=0$ for $i>n$. Accordingly, $H_i (\tilde{X}_j)=0$ for $i>n$). As a result, ${\rm im}\big(\pi_1 (u_{X_{j+1}})\big)$ and ${\rm im}\big(H_i (\tilde{u}_{\tilde{X}_{j+1}})\big)$ are retracts of $\pi_1 (X_j)$ and $H_i (\tilde{X}_j)$, $2\leq i\leq n$, respectively.
 
Assume that ${\rm im}\big(\pi_1 (u_{X_{j+1}})\big)=\pi_1 (X_j)$ and  ${\rm im}\big( H_i (\tilde{u}_{\tilde{X}_{j+1}})\big)=H_i (\tilde{X}_j)$, for $2\leq i\leq n$. Then $\pi_1 (u_{X_{j+1}})$ and $H_i (\tilde{u}_{\tilde{X}_{j+1}})$ are isomorphisms,  for $2\leq i\leq n$. Accordingly, $\pi_1 (d_{X_{j+1}})$ and $H_i (\tilde{d}_{\tilde{X}_{j+1}})$ are isomorphisms,  for $2\leq i\leq n$. This implies that $d_{X_{j+1}}:X_j \to X_{j+1}$ is a homotopy equivalence by the Whitehead Theorem (if $X$ and $Y$ are CW-complexes and there exists a map $f : X \to Y$ such that $f$ induces an isomorphism $\pi_1 (f) : \pi_1 (X) \to \pi_1(Y )$ and isomorphisms $H_i (\tilde{f}):H_i (\tilde{X})\to H_i (\tilde{Y})$ for all $i\in \mathbb{N}$, then $f$ is a homotopy equivalence) which is a  contradiction to  $X_{j+1} <X_j$. Therefore, either ${\rm im}\big(\pi_1 (u_{X_{j+1}})\big)$ is a proper retract of $\pi_1 (X_j)$ or  ${\rm im}\big( H_{i_0} (\tilde{u}_{\tilde{X}_{j+1}})\big)$ is a proper retract of $H_{i_0} (\tilde{X}_j)$, for some $2\leq i_0 \leq n$. So by Lemma \ref{lem1}, either $sl \Big({\rm im}\big(\pi_1 (u_{X_{j+1}})\big) \Big)<sl(\pi_1 (X_j))$ or $sl\Big( {\rm im}\big( H_{i_0} (\tilde{u}_{\tilde{X}_{j+1}})\big) \Big) <sl(H_{i_0} (\tilde{X}_j))$ for some $2\leq i_0 \leq n$. Since $\pi_1 (X_{j+1})\cong {\rm im}\big( \pi_1 (u_{X_{j+1}})\big)$ and $H_{i_0} (\tilde{X}_{j+1})\cong {\rm im}\big( H_{i_0} (\tilde{u}_{\tilde{X}_{j+1}})\big)$, we get by Lemma \ref{must} that either $sl(\pi_1 (X_{j+1}))<sl(\pi_1 (X_j))$ or $sl(H_{i_0} (\tilde{X}_{j+1}))<sl(H_{i_0} (\tilde{X}_j))$, for some $2\leq i_0 \leq n$. Note that by Lemma \ref{lem1}, in general we have  $sl(\pi_1 (X_{j}))\leq sl(\pi_1(P))<\infty$ and $sl(H_{i} (\tilde{X}_{j}))\leq sl(H_i(\tilde{P}))<\infty$ for all $i$. Then for $j_0:=sl(\pi_1 (P)) +\sum_{i=2}^{n} sl( H_i (\tilde{P}))$  we have $sl(\pi_1 (X_{j_0}))=0$ and $sl(H_{i} (\tilde{X}_{j_0}))=0$ for all $i$. Hence by Lemma \ref{must}, $\pi_1 (X_{j_0})=1$ and $H_{i} (\tilde{X}_{j_0})=0$ for all $i$. Thus $X_{j_0}$ is homotopically trivial, and so  the proof is finished. 
\end{proof}
\begin{lemma}\label{lem2}
Let $G$ be a finite group and $M$ be a finitely generated $\mathbb{Z}G$-module. Then $M$ is a finitely generated abelian group.
\end{lemma}
\begin{proof}
Note that if $\{ m_1 ,\ldots ,m_k \}\subset M$ is a finite generating set for $M$ as a $\mathbb{Z}G$-module, then by enumerating $G=\{g_1,...,g_L\}$,  $\{g_l \cdot m_k \mid l=1,\ldots,L, \, k=1,\ldots,K\}$ is a finite generating set for $M$ as an abelian group.
\end{proof}
Ko\l{}odziejczyk in \cite{4} showed that polyhedra with finite fundamental group have finite capacity, and hence finite depth. Next corollary presents an upper bound for the depth of such polyhedra.
\begin{corollary}\label{finite}
Let $P$ be a polyhedron of dimension $n$  with finite fundamental group. Then $D(P)\leq  sl\big( \pi_1 (P)\big) +\sum_{i=2}^{n}n_i$, where $n_i$ is the number of nonzero direct summands  in the canonical form of $H_i (\tilde{P})$.
\end{corollary}
\begin{proof}
Recall that for every finite CW-complex $L$, the $k$-chains in the cellular complex of the universal cover $\tilde{L}$, $C_k (\tilde{L})$, have a structure of a finitely generated module over $\mathbb{Z}\pi_1 (L)$ with the basis corresponding to the $k$-cells of $L$ (see \cite{Wh} or \cite[Chapter 2, p. 28]{Co}). Since the $k$-cycles, $Z_k (\tilde{L})$, is a submodule of $C_k (\tilde{L})$, then $Z_k (\tilde{L})$ is a finitely generated $\mathbb{Z}\pi_1 (L)$-module. Thus $H_k (\tilde{L})$, a quotient module of $Z_k (\tilde{L})$, is also a finitely generated $\mathbb{Z}\pi_1 (L)$-module.

Therefore the $H_i (\tilde{P})$, for $i=1,2,\ldots $, are all finitely generated as  $\mathbb{Z}\pi_1 (P)$-module. So by the hypothesis and Lemma \ref{lem2}, $H_i (\tilde{P})$'s are all finitely generated abelian group for $i\geq 2$. Thus the proof is complete by Proposition \ref{pro2} and Theorem \ref{MainTheorem}.
\end{proof}
The following corollary is an immediate consequence of the above result.
\begin{corollary}\label{simply}
If  $P$ is a simply connected polyhedron of dimension $n$, then $D(P)\leq  \sum_{i=2}^{n}n_i$, where $n_i$ is the number of nonzero direct summands  in the canonical form of $H_i (P)$.
\end{corollary}
In the next three  corollaries, we provide upper bounds for the depth of some classes of finite polyhedra with abelian, free or elementary amenable fundamental group. Their proofs are deduced from Proposition \ref{pro2} and Theorem \ref{MainTheorem}.

Recall that if $P$ is a connected and finite polyhedron, then $\pi_1 (P)$ is a finitely presented group (see \cite[Corollary 7.37]{R}).
\begin{corollary}\label{abelian}
Let $P$ be  a polyhedron of dimension $n$ with abelian $\pi_1 (P)$  and  finitely generated  $H_i (\tilde{P})$, for $i\geq 2$. Then  $D(P)\leq \sum_{i=1}^{n}n_i$, where $n_1$ and $n_i\; (i\geq 2)$ are the number of nonzero direct summands  in the canonical form of $\pi_1 (P)$ and $H_i (\tilde{P})$, respectively. 
\end{corollary}

\begin{corollary}\label{free}
Let $P$ be a polyhedron of dimension $n$ with free $\pi_1 (P)$  and  finitely generated  $H_i (\tilde{P})$, for $i\geq 2$. Then  $D(P)\leq rank(\pi_1 (P)) +\sum_{i=2}^{n}n_i$, where $n_i$ is the number of nonzero direct summands  in the canonical form of $H_i (\tilde{P})$.
\end{corollary}
\begin{corollary}\label{amen}
Let $P$ be  a polyhedron of dimension $n$ with elementary amenable  $\pi_1 (P)$ of finite cohomological dimension  and  finitely generated  $H_i (\tilde{P})$, for $i\geq 2$. Then  $D(P)\leq h(\pi_1 (P)) +\sum_{i=2}^{n}n_i$, where $h(\pi_1 (P))$ is the Hirsch length of $\pi_1 (P)$ and  $n_i$ is the number of nonzero direct summands  in the canonical form of $H_i (\tilde{P})$.
\end{corollary}
Recently, in \cite{D2018}, Ko\l{}odziejczyk  proved that  2-dimensional polyhedra whose fundamental groups are elementary amenable with finite cohomological dimension have finite depth.  In the sequel, we are going to present upper bounds for such polyhedra. 

Recall that if $P$ is a polyhedron of dimension $n$, then $H_n (P)$ is free abelian (see, for example, \cite[Theorem 7.24]{R}).  Now we  state our second main result.
\begin{theorem}\label{MainTheorem1}
If $P$ is a 2-dimensional polyhedron and $sl(\pi_1 (P))<\infty$, then $D(P)\leq  sl\big(\pi_1 (P)\big)+rank \big( H_2 (P)\big)$.
\end{theorem}
\begin{proof}
Consider the following chain of CW-complexes:
$$
\cdots < X_{i+1}<X_i <\cdots <X_3 <X_2 <X_1 <X_0 =P.
$$
 Let $d_{X_{i+1}}:X_i \to X_{i+1}$ and $u_{X_{i+1}} :X_{i+1}\to X_i$ be the domination of $X_i$ over $X_{i+1}$ and the converse map, i.e., $d_{X_{i+1}} u_{X_{i+1}}\simeq id_{X_{i+1}}$. Then $\pi_1 (d_{X_{i+1}}) \pi_1 (u_{X_{i+1}})= id_{\pi_1 (X_{i+1})}$ and  $H_2 (d_{X_{i+1}}) H_2 (u_{X_{i+1}})= id_{H_2 (X_{i+1})}$. As a result, ${\rm im}\pi_1 (u_{X_{i+1}})$ and ${\rm im}H_2 (u_{X_{i+1}})$ are retracts of $\pi_1 (X_i)$ and $H_2 (X_i)$.  
 
Assume that ${\rm im}\big(\pi_1 (u_{X_{i+1}})\big)=\pi_1 (X_i)$ and  ${\rm im}\big(H_2 (u_{X_{i+1}})\big)=H_2 (X_i)$. Then $\pi_1 (u_{X_{i+1}})$ and $H_2 (u_{X_{i+1}})$ are isomorphisms. Accordingly, $\pi_1 (d_{X_{i+1}})$ is an isomorphism and $H_2 (X_i)\cong H_2 (X_{i+1})$. since $\pi_1 (X_i)\cong \pi_1 (X_{i+1})$, we have  $H_1 (X_i)\cong H_1 (X_{i+1})$ by the Hurewicz Theorem (see \cite[Theorem 4.29]{R}). This fact and  $H_2 (X_i)\cong H_2 (X_{i+1})$ imply that $\chi (X_i)=\chi (X_{i+1})$, where $\chi (X)$ denotes the Euler-Poincare  characteristic of polyhedron $X$ of dimension $m$. Now since $d_{X_{i+1}}:X_i \to X_{i+1}$ is a homotopy domination between $X_i$ and $X_{i+1}$ that induces an isomorphism on the fundamental groups,  it is a homotopy equivalence (by \cite[Theorems 2]{D2018}) which contradicts  $X_{i+1} <X_i$. Therefore, either ${\rm im}\big(\pi_1 (u_{X_{i+1}})\big)$ is a proper retract of $\pi_1 (X_i)$ or  ${\rm im}\big(H_2 (u_{X_{i+1}})\big)$ is a proper retract of $H_2 (X_i)$.  So by Lemma \ref{lem1}, $sl\Big({\rm im}\big(\pi_1 (u_{X_{i+1}})\big)\Big)<sl\big( \pi_1 (X_i)\big)$ or $rank\Big( {\rm im}\big(H_2 (u_{X_{i+1}})\big)\Big)<rank\big( H_2 (X_i)\big)$. Since $\pi_1 (X_{i+1})\cong {\rm im}\big(\pi_1 (u_{X_{i+1}})\big)$ and $H_2 (X_{i+1})\cong {\rm im}\big(H_2 (u_{X_{i+1}})\big)$, we have by Lemma \ref{must} that   $sl\big( \pi_1 (X_{i+1})\big)<sl\big(\pi_1 (X_i)\big)$ or $rank\big( H_2 (X_{i+1})\big)<rank\big( H_2 (X_i)\big)$. Thus for $i_0:=sl\big(\pi_1 (P)\big)+rank \big( H_2 (P)\big)$, we have $sl(\pi_1 (X_{i_0}))=0$ and $rank(H_2 (X_{i_0}))=0$ which means by Lemma \ref{must} that $\pi_1 (X_{i_0})=1$ and $H_2 (X_{i_0})=0$, Hence,  $X_{i_0}$ is homotopically trivial and so  the proof is finished. 
\end{proof}
In the next three  corollaries, we present upper bounds for the depth of 2-dimensional polyhedra with abelian, free or elementary amenable fundamental group of finite cohomological dimension. Their proofs are deduced from Proposition \ref{pro2} and Theorem \ref{MainTheorem1}.
\begin{corollary}\label{amen2}
If $P$ is a 2-dimensional polyhedron with elementary amenable fundamental group of finite cohomological dimension, then $D(P)\leq h\big( \pi_1 (P)\big) +rank\big( H_2(P)\big)$.
\end{corollary}
\begin{corollary}\label{abelian2}
If $P$ is a 2-dimensional polyhedron with abelian fundamental group, then $D(P)\leq n +rank\big( H_2(P)\big)$, where $n$ is  the number of nonzero direct summands  in the canonical form of $\pi_1 (P)$.
\end{corollary}
 Ko\l{}odziejczyk  \cite{D2018} showed that 2-dimensional polyhedra whose fundamental groups are limit groups have finite depth (see \cite[Corollary 3]{D2018}). The class of limit groups  which is a generalization of the class of free groups are always finitely presented. In the following, we present an upper bound for 2-dimensional polyhedra with free fundamental group.
\begin{corollary}\label{free2}
If $P$ is a 2-dimensional polyhedron with free fundamental group, then $D(P)\leq rank\big( \pi_1 (P)\big) +rank\big( H_2(P)\big)$.
\end{corollary}
\section{Examples}
In this section, we provide some examples to show that some of upper bounds presented in the previous section are sharp. 
\begin{example}
Consider  $P=\bigvee_{i\in I} (\vee_{r_n} S^i)$, where $\vee_{r_i} S^i$ denotes the wedge sum of $r_i$ copies of $S^i$, $I$ is a finite subset of $\mathbb{N}\setminus \{ 1\}$ and $r_i \in \mathbb{N}$. Then $P$ is a simply connected polyhedron with   $H_i (P)\cong \mathbb{Z}^{r_i}$, for $i\geq 2$. Hence,  $n_i=r_i$ for $i\geq 2$, and so by Corollary \ref{simply}, $D(P)\leq \sum_{i\in I}r_i $. 

Note that by \cite[Theorems 3.3]{Mo}, every topological space homotopy dominated by $\bigvee_{i\in I} (\vee_{r_i} S^i)$ has the same homotopy type of a CW-complex of the form $\bigvee_{i\in I} (\vee_{s_i} S^i)$, where  $0\leq s_i \leq r_i$. This implies that $D(P)=\sum_{i\in I}r_i $ since the system
 $$
1<S^2 <S^2 \vee S^2 <\cdots  <\vee_{r_2} S^2<\vee_{r_2} S^2 \vee S^3<\cdots <\bigvee_{i\in I} (\vee_{r_i} S^i)  = P
 $$
is a chain of length $\sum_{i\in I}r_i $ for $P$. This fact demonstrates that the upper bound presented in Corollary \ref{finite} (also Corollary \ref{simply}) is sharp. 
\end{example}
\begin{example}
Let $P=\prod_{n\in I}(\prod_{r_n}S^n)$, where $\prod_{r_n}S^n$ denotes the product of $r_n$ copies of $S^n$, $I$ is a finite subset of $\mathbb{N}$ and $r_n \in \mathbb{N}$. Then $P$ is a finite polyhedron with free abelian $\pi_1 (P)\cong \mathbb{Z}^{r_1}$ and that   $\tilde{P}=\prod_{n\in I\setminus \{ 1\}}(\prod_{r_n}S^n)\times \mathbb{R}^{r_1}\simeq \prod_{n\in I\setminus \{ 1\}}(\prod_{r_n}S^n)$.  By the K\"{u}nneth formula, each $H_i (\tilde{P})$ is a free abelian group, so we only need to keep track of the rank. In general, the rank of $H_i (\tilde{P})$ (which is $n_i$) is the coefficient of $x^i$ in the Poincare polynomial $\prod_{n\in I\setminus \{ 1\}}(1+x^{n})^{r_n}$ of $\tilde{P}$,  for $i\geq 2$. Hence, by Corollary \ref{abelian}   we have $D(P)\leq r_1 + \sum_{i=2}^{\dim (P)}n_i$, where  $n_i$   is the coefficient of $x^i$ in $\prod_{n\in I\setminus \{ 1\}}(1+x^{n})^{r_n}$, for $i\geq 2$. 

For instance, take $P=S^1 \times S^n$, for $n\geq 2$. On the one hand, by the above argument, $D(P)\leq 2$. On the other hand, by \cite[Theorems 3.6]{Mo1},  $1,S^1 ,S^n$, and $S^1 \times S^n$ are the only spaces homotopy dominated by $P$, and so  $D(P)=2$. This shows that the upper bounds presented in Corollaries \ref{finite} and \ref{abelian} are sharp. 
\end{example}
\begin{example}
Consider 2-dimensional polyhedra $P_1=S^1   \vee S^2$ and $P_2 =S^1 \vee S^1  \vee S^2$. We have  $\pi_1 (P_1)\cong\mathbb{Z}$, $\pi_1 (P_2)\cong\mathbb{Z}\ast \mathbb{Z}$ and $H_2 (P_1)\cong H_2 (P_2)\cong \mathbb{Z}$. Then by Corollaries \ref{abelian2} and \ref{free2}, respectively, we have $D(P_1)\leq 2$ and $D(P_2)\leq 3$. Moreover, from \cite[Theorems 3.3]{Mo} we have $D(P_1)=2$ and $D(P_2)=3$ which shows that the upper bounds given in Corollaries \ref{abelian2} and \ref{free2} are sharp. Note that  $H_2 (\tilde{P}_1)$ and $H_2 (\tilde{P}_2)$ are infinitely generated abelian groups, and so we could not apply Corollaries \ref{abelian} and Corollary \ref{free}  for $P_1$ and $P_2$, respectively.  
\end{example}

\section{Statements and Declarations}
\textbf{Conflict of Interest} All authors state no conflict of interest.

\textbf{Data availability statement} Data sharing is not applicable to this article as no data sets were generated or analysed.

\end{document}